\documentclass[12pt]{amsart}
\usepackage{graphicx} 
\usepackage{amsmath,amsthm,amssymb,epsfig}
\usepackage[table]{xcolor}
\usepackage{url}
\usepackage{tikz}
\usepackage{float}
\usepackage{mathtools}
\usepackage{subcaption}
\usepackage[linesnumbered,lined,boxed,commentsnumbered]{algorithm2e}
\usepackage{MnSymbol}
\usepackage{xcolor}
\usepackage{comment}
\usepackage{tkz-graph}
\usetikzlibrary{fit}
\usetikzlibrary{shapes}
\usepackage{enumitem}
\usepackage[left=3.5cm,top=3.5cm,right=3.5cm,bottom=3.5cm]{geometry}
\usepackage{mathrsfs}

\newcommand{\calB}{\mathcal{B}}
\newcommand{\calN}{\mathcal{N}}

\newcommand{\IG}{\mathrm{IG}}
\newcommand{\core}{\mathrm{core}}

\newtheorem{theorem}{Theorem}
\newtheorem{corollary}[theorem]{Corollary}
\newtheorem{lemma}[theorem]{Lemma}

\title{Hypergraph burning, matchings, and zero forcing}

\author[A.\ Bonato]{Anthony Bonato}
\author[C.\ Jones]{Caleb Jones}
\author[T.G.\ Marbach]{Trent G.\ Marbach}
\author[T.\ Mishura]{Teddy Mishura}
\author[Z.\ Zhang]{Zhiyuan Zhang}
\address[A1,A2,A3,A4,A5]{Toronto Metropolitan University, Toronto, Canada}

\email[A1]{(A1) abonato@torontomu.ca}
\email[A2]{(A2) caleb.w.jones@torontomu.ca}
\email[A3]{(A3) trent.marbach@torontomu.ca}
\email[A4]{(A4) tmishura@torontomu.ca}
\email[A5]{(A5) zhiyuan.zhang@torontomu.ca}


\begin{document}

\begin{abstract} 
Lazy burning is a recently introduced variation of burning where only one set of vertices is chosen to burn in the first round. In hypergraphs, lazy burning spreads when all but one vertex in a hyperedge is burned. The lazy burning number is the minimum number of initially burned vertices that eventually burns all vertices. We give several equivalent characterizations of lazy burning on hypergraphs using matchings and zero forcing, and then apply these to give new bounds and complexity results. 

We prove that the lazy burning number of a hypergraph $H$ equals its order minus the maximum cardinality of a certain matching on its incidence graph. Using this characterization, we give a formula for the lazy burning number of a dual hypergraph and give new bounds on the lazy burning number based on various hypergraph parameters. We show that the lazy burning number of a hypergraph may be characterized by a maximal subhypergraph that results from iteratively deleting vertices in singleton hyperedges. 

We prove that lazy burning on a hypergraph is equivalent to zero forcing on its incidence graph and show an equivalence between skew zero forcing on a graph and lazy burning on its neighborhood hypergraph. As a result, we show that finding an upper bound on the lazy burning number of a hypergraph is NP-complete, which resolves a conjecture from \cite{BJR}. By applying lazy burning, we show that computing an upper bound on the skew zero forcing number for bipartite graphs is NP-complete. We finish with open problems.

\end{abstract}
\maketitle
\section{Introduction}

Graph burning is a simplified model for the spread of influence in a network. Associated with the process is a parameter introduced in \cite{BJR,BJR1}, the burning number, which quantifies the speed at which the influence spreads to every vertex. Given a graph $G$, the burning process on $G$ is a discrete-time process. At the beginning of the first round, all vertices are unburned. In each round, first, all unburned vertices that have a burned neighbor become burned, and then one new unburned vertex is chosen to burn if such a vertex is available. If at the end of round $k$ every vertex of $G$ is burned, then $G$ is $k$-\emph{burnable}. The \emph{burning number} of $G,$ written $b(G),$ is the least $k$ such that $G$ is $k$-burnable. For further background on graph burning, see the survey \cite{survey} and the book \cite{bbook}.

\emph{Hypergraph burning} was introduced in \cite{BJP,JT} as a natural variant of graph burning. The rules for hypergraph burning are identical to those of burning graphs, except for how burning propagates within a hyperedge. Recall that in a hypergraph, a \emph{singleton} edge contains exactly one vertex. In hypergraphs, the burning spreads to a vertex $v$ in round $r$ if and only if there is a non-singleton hyperedge $\{v,u_1,\ldots,u_k\}$ such that $v$ was not burned and each of $u_1,u_2,\ldots,u_k$ was burned at the end of round $r-1$. In particular, a vertex becomes burned if it is the only unburned vertex in a hyperedge. A natural variation of burning that is our principal focus here is \emph{lazy hypergraph burning}, where a set of vertices is chosen to burn in the first round only; no other vertices are chosen to burn in later rounds.  The \textit{lazy burning number} of $H$, denoted $b_L(H)$, is the minimum cardinality of a set of vertices burned in the first round that eventually burn all vertices. Note that for a hypergraph $H,$ we have that $b_L(H) \le b(H).$ We refer to the set of vertices chosen to burn in the first round as a \emph{lazy burning set}. A lazy burning set $S$ such that $|S|=b_L(H)$ is called \emph{optimal}. If a vertex is burned and not part of a lazy burning set, then we say it is burned by \emph{propagation}; all rounds other than the first where a lazy burning set is chosen are called \emph{propagation rounds}.

We note that we consider a slightly modified rule for lazy burning than the one given in \cite{BJP}, which simplifies our discussion. If a hyperedge $h$ contains exactly one unburned vertex $v$, then $v$ burns. The only difference between this rule and the original one is that now singleton hyperedges \emph{spontaneously burn} in the sense that they cause the vertices they contain to burn immediately. Note that this cannot increase the lazy burning number. Further, no optimal lazy burning set will contain a vertex belonging to a singleton edge, and various bounds (such as $|V(H)|-|E(H)|\leq b_L(H)$ from \cite{BJP}) remain unchanged with only minor variations in their proof.

In the \emph{zero forcing} coloring process, a vertex subset of a graph $G$ is initially chosen to be black, while the remaining vertices are white (some authors alternatively refer to the colors as blue and white, respectively). Consider the following \emph{zero forcing color change rule}: a black vertex $u$ can change the color of a white vertex $v$ if $v$ is the only white neighbor of $u$; we say that $u$ \emph{forces} $v$ or that $v$ is \emph{forced}, and write $u \rightarrow v.$
A \emph{zero forcing set} of $G$ is a set $Z \subseteq V(G)$ of vertices such that if the vertices of $Z$ are black and the rest are white, then every vertex eventually becomes black after repeated applications of the zero forcing color change rule. The minimum cardinality of a zero forcing set for a graph is known as the \textit{zero forcing number} $z(G)$ of the graph $G$. The {\em skew zero forcing color change rule} requires that if a vertex $u$ has exactly one white neighbor $v$, then $u$ forces $v$. The only difference between the two processes is that in skew zero forcing, a vertex does not need to be black to force one of its neighbors. A {\em skew zero forcing set} and the {\em skew zero forcing number $z_0(G)$} are defined analogously.  As shown in \cite{IMA}, we have that $z_0(G)\leq z(G)$. For more background on these coloring processes, see \cite{AIM, FHsurvey, hlsbook, IMA}. 

The paper is organized as follows. In Section~2, we define chronological lists of lazy burning on a hypergraph, which leads to the notion of a $C$-matching. In Corollary~\ref{cor:Burning_equals_CMatching}, we prove that $b_L(H)=|V(H)|-m(H)$, where $m(H)$ is the maximum cardinality of a $C$-matching on the incidence graph of $H$. This provides a formula for the lazy burning number of the dual hypergraph in Corollary~\ref{cor1} and provides various new lower and upper bounds on the lazy burning number. For example, Corollary~\ref{cor:delta-delta} gives that in a linear hypergraph with minimum degree $\delta$ with at least $\delta$ hyperedges, $b_L(H) \geq |V(H)| -|E(H)| + \binom{\delta}{2}$, while Theorem~\ref{upp} shows that $b_L(H) \leq |V(H)| - \left\lceil \frac{|E(H)|}{\Delta}\right\rceil,$ where $\Delta$ is the maximum vertex degree. In Section~3, we characterize the lazy burning number of a hypergraph using its core, which is the hypergraph that results by iteratively deleting vertices in singleton hyperedges; see Theorem~\ref{thm: empty core}. Section~4 focuses on connections with zero forcing, and Theorem~\ref{zero_set_equivalence} shows that lazy burning is equivalent to zero forcing on incidence graphs. We prove in Theorem~\ref{skew_forcing_equals_lazy_burning} that skew zero forcing on a graph is equivalent to lazy burning on its neighborhood hypergraph. Using this result, we prove that computing an upper bound on the lazy burning number of a hypergraph is NP-complete, which answers Problem 11 from \cite{BJP}. In Theorem~\ref{szf on bpt is npc}, we find applications of lazy burning to skew zero forcing and show the skew zero forcing problem on bipartite graphs is NP-complete. The final section lists open problems.

For more background on hypergraphs, see \cite{berge1,berge2,bretto}, and for more background on graph theory, see~\cite{west}.

\section{Chronological Lists and $C$-Matchings}

A concept gaining in popularity in the literature is a \emph{chronological list of zero forcing}.  Given an initial zero forcing set $S_0$, a list of forces $(u_1 \rightarrow v_1, u_2 \rightarrow v_2, \ldots, u_k \rightarrow v_k)$ is defined so that the list of sets $S_0, S_1, S_2, \ldots, S_k$ with $S_{i+1} = S_i \cup \{v_{i+1}\}$ satisfies $N[u_i]\setminus\{v_i\} \subseteq S_{i-1}$. Intuitively, if the set $S_{i-1}$ is a subset of the black vertices at any time during the zero forcing process, then the vertex $u_i$ will force the vertex $v_i$ in the next round if $v_i$ is not already black. It follows that $S_i$ will be a subset of the black vertices at some time in the process. We note that we may have two vertices $v_i$ and $v_j$, with $ v_i$ occurring earlier than $v_j$ in the chronological list, but where vertex $v_j$ would be forced by some black vertex in the zero forcing process in an earlier round than $v_i$. A chronological list of skew zero forces may be defined analogously. If the context is clear, then we abbreviate the lists of either type as a chronological list of forces. 

We illustrate chronological lists in Figure~\ref{fig:ZF_on_C}, where two adjacent black vertices, $a$ and $b$, are initially chosen as our zero forcing set.  
In the first zero forcing propagation round, $c$ and $d$ are forced by $a$ and $b,$ respectively. In the second round, $e$ is forced by either $c$ or $d$. 
However, one possible chronological list of forces is $(b \rightarrow d, d \rightarrow e, e \rightarrow c)$, so $c$ is the last to be forced in the chronological list of forces, even though $c$ is forced before $e$ in the zero forcing process. 

\begin{figure}[htpb!]
    \centering
\includegraphics[width=\linewidth/2]{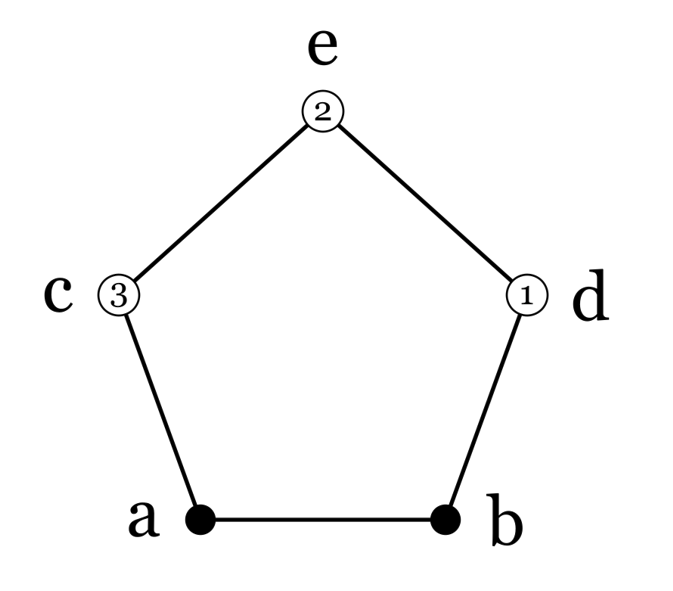}
    \caption{A zero forcing set for the $5$-cycle, with a chronological list of zero forces given by the vertices labeled 1, 2, and 3.}
    \label{fig:ZF_on_C}
\end{figure}

We now introduce an analogous concept for the lazy burning of a hypergraph $H$. 
A \emph{chronological list of lazy burnings}, $(\{h_1,v_1\},\{h_2,v_2\}, \ldots, \{h_k,v_k\})$ with $v_i\in h_i \in E(H)$, is such that the list of sets $$B_0, B_1=B_0\cup\{v_1\}, B_2=B_1\cup \{v_2\}, \ldots, B_k=B_{k-1}\cup\{v_k\}$$ with $B_{i+1} = B_i\cup \{v_{i+1}\}$ satisfying $h_i\setminus\{v_i\} \subseteq B_{i-1}$. 
We note that, if $B_k = V(H)$, then $B_0$ is a burning set for $H$ and $k = n - |B_0|$. 
In the case that we take $B_0$ to be a lazy burning set, if $B_{i-1}$ is burned at any time during the lazy burning process, then each vertex in $h_i\setminus \{v_i\}$ is burned. Hence, the vertex $v_i$ will be burned in the next round if it is not already burned. It follows that the entire hypergraph will be burned at some point and that vertex $v_i$ will be burned by round $i$ within the lazy burning process on $H$. We summarize these observations in the following lemma.

\begin{lemma} \label{lem:ChronList_equals_Burning}
If $(\{h_1,v_1\},\{h_2,v_2\}, \ldots, \{h_k,v_k\})$ is a chronological list of lazy burnings of $H$, then 
$V(H) \setminus \{v_1, v_2, \ldots, v_k\}$ is a lazy burning set for $H$, and the vertex $v_i$ will be burned during or before the $i$-th lazy burning propagation round. 
\end{lemma}

A \emph{matching} in a graph is a set of edges such that no two edges in the set share a common vertex. 
The \emph{incidence graph} $\IG(H)$, also known as a \emph{Levi graph}, of a hypergraph $H$ is the bipartite graph with vertex set $V(H)\cup E(H)$, and with an edge in $\mathrm{IG}(H)$ between vertex $v\in V(H)$ and vertex $h\in E(H)$ whenever $v\in h$. We note that there is an implicit ordering on the parts $V(H)$ and $E(H)$ of the bipartition. 
We will write an edge of $\IG(H)$ as an ordered pair, say $vh$, where $v\in V(H)$ and $h \in E(H)$. 

A chronological list of lazy burning on a hypergraph $H$ can be converted to a matching in $\mathrm{IG}(H)$, although we need a more restrictive structure than this for our purposes. We define a \emph{$C$-matching} in $\mathrm{IG}(H)$ to be a list of edges from $\mathrm{IG}(H)$, say $M=(v_1h_1, v_2h_2, \ldots, v_kh_k)$, such that
$$ h_i \cap V_{i}=\emptyset,$$
where $V_{i} = \{v_{i+1}, v_{i+2}, \ldots, v_{k}\}$ and $1\leq i\leq k$. 
Define $m(H)$ to be the maximum cardinality of a $C$-matching on $\mathrm{IG}(H)$.

\begin{figure}[h!]
    \centering
    \begin{subfigure}{0.40\textwidth}
        \centering
\includegraphics[width=\linewidth]{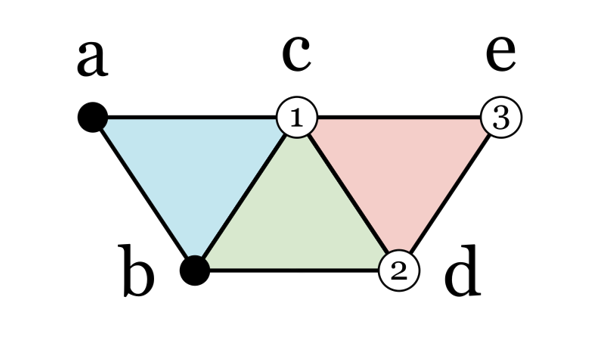}
        \caption{A hypergraph with hyperedges $\{a,b,c\},\{b,c, d\},\{c,d,e\}$.\ A lazy burning set $\{a,b\}$ is indicated by the filled vertices, and a chronological list of lazy burnings is indicated by the vertices containing numbers.}
        \label{fig:image1}
    \end{subfigure}
    \hfill
    \begin{subfigure}{0.50\textwidth}
        \centering
        \includegraphics[width=\linewidth]{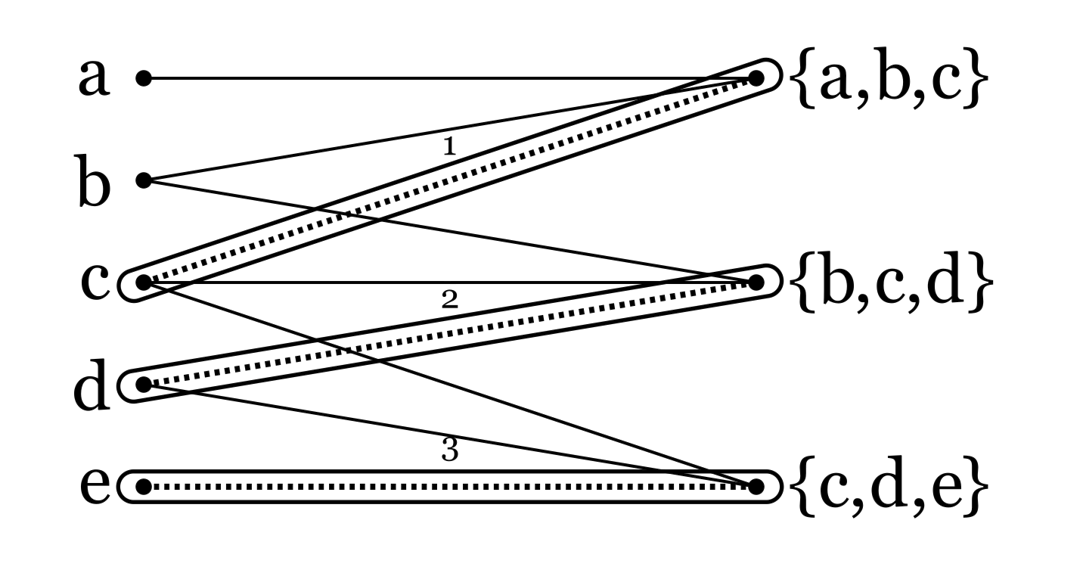}
        \caption{The corresponding incidence graph of the hypergraph.\ A $C$-matching is indicated by the circled edges, labeled from 1 to 3 based on their appearance in the $C$-matching.}
        \label{fig:image2}
    \end{subfigure}
    \caption{Comparison of a chronological list for lazy burning (A) to a $C$-matching on the corresponding incidence graph (B).}
    \label{fig:two_images}
\end{figure}

As the following lemma demonstrates, chronological lists and $C$-matchings are equivalent.

\begin{lemma} \label{lem:ChronList_equals_CMatching}
The list $(\{h_1,v_1\},\{h_2,v_2\}, \ldots, \{h_k,v_k\})$ is a chronological list of lazy burning on $H$ if and only if $(v_1h_1, v_2h_2, \ldots, v_kh_k)$ is a $C$-matching on $\mathrm{IG}(H)$. 
\end{lemma}
\begin{proof}
For the forward implication, suppose that $$(\{h_1,v_1\},\{h_2,v_2\}, \ldots, \{h_k,v_k\})$$ is a chronological list of lazy burning on $H$. 
For $ 0\leq i\leq  k $, we define $V_{i} = \{v_{i+1}, v_{i+2}, \ldots, v_{k}\}$ and $B_{i} = V(H) \setminus V_{i}$. 
From the definition of a chronological list, for each $1\leq i\leq k$, $h_i\setminus \{v_i\} \subseteq B_{i-1} = V(H) \setminus V_{i-1}$. 
This is equivalent to $(h_i\setminus \{v_i\})\cap V_{i-1} = \emptyset,$ which itself is equivalent to $h_i \cap V_{i} = \emptyset$. This is what is required to make $(v_1h_1, v_2h_2, \ldots, v_kh_k)$ a $C$-matching of $H$. 

By reversing the chain of implications in the previous paragraph, the proof of the reverse direction follows. 
\end{proof}

Lemmas~\ref{lem:ChronList_equals_Burning} and \ref{lem:ChronList_equals_CMatching} can then be combined to yield the following equality characterizing the lazy burning number of a hypergraph in terms of its order and a maximum cardinality of a $C$-matching. 
For a $C$-matching $M = (v_1h_1, v_2h_2,\dots, v_mh_m)$, define $M_V= \{v_1,v_2,\dots, v_m\}$, which is the set of vertices that occur in an edge of $M$.

\begin{corollary} \label{cor:Burning_equals_CMatching}
Let $H$ be a hypergraph and $M$ be a $C$-matching in $\IG(H)$. We have that $V(H)\setminus M_V$ is a lazy burning set for $H$, and $$m(H) = |V(H)| - b_L(H).$$
\end{corollary}

Define the \emph{dual} of a hypergraph $H=(V,E)$ to be the hypergraph $H^*$ with vertex set $E$ and with hyperedge set $\{N(v) : v \in V\}$, where we use $N(v)$ to denote the set of hyperedges in $E$ that contain vertex $v$. 
The bipartite graph formed by dropping the order of the parts of $\mathrm{IG}(H)$ is the same as that from dropping the order of the parts of $\mathrm{IG}(H^*)$. 

Given a $C$-matching $M=(v_1h_1, v_2h_2, \ldots, v_kh_k)$ in $\mathrm{IG}(H)$, the \emph{retrograde of $M$}, written $M^R$, is the list of edges of $M$ in reverse order and with the ordering of the vertices in each edge swapped; that is, $(h_kv_k, h_{k-1}v_{k-1}, \ldots, h_1v_1)$.  The retrograde of $M$ is an ordered set of ordered edges of $\IG(H^*)$. 
We will show that this retrograde is a $C$-matching of $\IG(H^*)$. 

\begin{lemma} \label{lem:retrograde_dual}
Let $H=(V,E)$ be a hypergraph. 
For a $C$-matching $M$ in a bipartite graph $\mathrm{IG}(H)$, the retrograde of $M$ is a $C$-matching in the bipartite graph $\mathrm{IG}(H^*)$. 
\end{lemma}
\begin{proof}
Let $M=(v_1h_1, v_2h_2, \ldots, v_mh_m)$ be a $C$-matching of the bipartite graph $(V\cup E,I)$, and define $V_i = \{v_{i+1}, v_{i+2}, \ldots, v_{m}\}$ and  $E_i = \{h_1, h_2, \ldots, h_{i-1}\}$. 
By the definition of incidence graphs and $C$-matchings, we have $h_i \cap V_{i}=\emptyset$. 
For contradiction, we suppose that the retrograde of $M$, $(h_mv_m, h_{m-1}v_{m-1}, \ldots, h_1v_1)$, is not a $C$-matching of the bipartite graph $\IG(H^*)=(E\cup V,\overline{I})$. 
Note that this implies that there is some $j$ such that $N(v_j) \cap E_{j} \neq \emptyset$, and so there exists some $i \leq j-1$ such that $h_i \in N(v_j)$.
However, this implies that $v_j \in h_i$. 
Since $j \geq i+1$, we also have that $v_j \in V_{i}$. 
It then follows that $v_j \in h_i \cap V_{i}$, but this contradicts the fact that $M$ is a $C$-matching. 
This completes the proof. 
\end{proof}

Lemma~\ref{lem:retrograde_dual} immediately yields the following result. 

\begin{corollary} \label{cor:CMatching_Duals}
    For a hypergraph $H$, $m(H) = m(H^*)$.
\end{corollary}

By combining Corollary~\ref{cor:Burning_equals_CMatching} and Corollary~\ref{cor:CMatching_Duals}, we derive the following. 
\begin{corollary}\label{cor1}
For a hypergraph $H$, $|V(H)| - b_L(H) = |V(H^*)| - b_L(H^*)$.
\end{corollary}
Alternatively, we may describe the lazy burning number of the dual hypergraph as 
    $ b_L(H^*)  = |E(H)| - |V(H)| + b_L(H)$.

\subsection{Lower Bounds}
Consider a $C$-matching $(v_1h_1, v_2h_2, \ldots, v_k h_k)$ in $\IG(H)$. 
The vertices in $\bigcup_{j \leq i} h_j$ must contain only vertices that are either in the lazy burning set for the corresponding lazy burning process or in $\{v_1, v_2, \ldots, v_i\}$. 
As such, describing the cardinalities of the hyperedges and the size of the intersection of hyperedges yields lower bounds on the lazy burning number. 

\begin{theorem}\label{thm:general_lazy_lowerbound}
    Let $H$ be a hypergraph such that each pair of hyperedges intersects in at most $\overline{\lambda}$ vertices. 
    For $t$ a positive integer, define $\overline{D}_t$ as the minimum sum of the cardinalities of $t$ distinct hyperedges in $H$.  
    For every $t$ with $1 \leq t \leq m(H)$, we have that 
    \[
    b_L(H) \geq 
    \overline{D}_t - \overline{\lambda} \binom{t}{2} - t.
    \]
\end{theorem}
\begin{proof}
Suppose $B$ is an optimal burning set for $H$. 
As seen in Corollary~\ref{cor:Burning_equals_CMatching}, we may define a $C$-matching in $\mathrm{IG}(H)$, say $M=(v_1h_1, v_2h_2, \ldots, v_mh_m)$, where $m = m(H)$ and $B = V(H) \setminus M_V$. 
 
Define $V_{i} = \{v_{i+1}, v_{i+2}, \dots, v_m\}$ for $1\leq i\leq m$. 
By the definition of a $C$-matching, we have that $$h_i \cap V_{i}=\emptyset.$$ 
Fix $t$ for $1\leq t\leq m$. 
The set $H_t=\left( \bigcup_{1 \leq j \leq t} h_j\right) \setminus\{v_1, v_2,\ldots, v_t\}$ satisfies $H_t\cap \{v_1, v_2,\ldots, v_m\} = \emptyset$ and, therefore, $H_t\subseteq B$. 
For $1\leq i\leq t$, define the set $h'_i = h_i \setminus \bigcup_{1 \leq j < i} h_{j}$, so that  $H'_t=\bigcup_{1 \leq j \leq t} h'_{j}$ is a union of disjoint sets.
Since each pair of hyperedges may intersect for at most $\overline{\lambda}$ vertices, we have that $|H'_t|\geq \sum_{j=1}^{t} (|h_{j}| - \overline{\lambda} (j-1))$. 
Also, $H_t = H'_t\setminus \{v_1, v_2, \ldots, v_m\} \subseteq H'_t\setminus \{v_1, v_2, \ldots, v_t\}$ and, thus, we have  
\begin{eqnarray*}
|H_t|   &\geq   & \left( \sum_{j=1}^{t} (|h_j| - \overline{\lambda} (j-1)) \right) - t \\
        &=      & \left( \sum_{j=1}^{t} |h_j|\right) - \overline{\lambda} \binom{t}{2} - t\\
        &\geq   & 
        \overline{D}_t
        - \overline{\lambda} \binom{t}{2} - t.
\end{eqnarray*}
    Since $H_t \subseteq B$, the proof is complete. 
\end{proof}

By Corollary~\ref{cor:Burning_equals_CMatching} and Corollary~\ref{cor:CMatching_Duals}, we may consider a dual argument that swaps the roles of the vertices and hyperedges in Theorem~\ref{thm:general_lazy_lowerbound}, which yields the following corollary.

\begin{corollary}
Let $H$ be a hypergraph such that each pair of vertices occurs together in at most $\lambda$ hyperedges and 
suppose that the sum of the degrees of the $t$ smallest-degree vertices is $D_t$.
%suppose that the degree of the vertices is $d_1, d_2, \ldots, d_{|V(H)|}$, where $d_i\leq d_{i+1}$ for all $i$. 
For each $t$ with $1 \leq t \leq m(H)$, 
    \[
    %b_L(H) \geq |V(H)| -|E(H)| + \left(\sum_{j=1}^{t} d_{j}\right) - \lambda \binom{t}{2} - t.
    b_L(H) \geq |V(H)| -|E(H)| + D_t - \lambda \binom{t}{2} - t.
    \]
\end{corollary}

If the graph has bounds placed on these parameters, then we may provide a more descriptive outcome.

\begin{theorem} \label{thm:BurningLowerUniformity}
    Let $H$ be a hypergraph such that each pair of hyperedges intersects in at most $\overline{\lambda}$ vertices, each hyperedge has cardinality at least $r$, and there are at least $\lceil \frac{r-1}{\overline{\lambda}} + \frac{1}{2} \rceil$ vertices. 
    We then have that
    $$b_L(H) \geq \frac{(r-1)^2}{2\overline{\lambda}} + \frac{r -1}{2} - \frac{3\overline{\lambda}}{8}.$$
\end{theorem}

\begin{proof}
From Theorem~\ref{thm:general_lazy_lowerbound}, we have that $b_L(H) \geq rt -  \overline{\lambda}\binom{t}{2} - t$ for $1\leq t\leq m(H)$. 
We may treat this lower bound as an integer-valued function $f(t)$ and then extend it to a real-valued function $f(x)$, where $x\in\mathbb{R}$. 
The function $f$ is maximized when the derivative is zero or at the endpoints of its domain. 
This occurs when $x=\frac{r-1}{\overline{\lambda}} + \frac{1}{2},$ since we have assumed that there are at least $\lceil \frac{r-1}{\overline{\lambda}} + \frac{1}{2} \rceil$ many vertices. 
As $t\in \mathbb{Z}$, the maximum must then occur at either $\lceil \frac{r-1}{\overline{\lambda}} + \frac{1}{2} \rceil$ or $\lfloor \frac{r-1}{\overline{\lambda}} + \frac{1}{2} \rfloor$. 
We may therefore assume that $t$ has the form $t = \frac{r-1}{\overline{\lambda}} + \frac{1}{2}+\varepsilon$ for some $-1 < \varepsilon < 1$. 
Substituting this value back in to $rt - \overline{\lambda}\binom{t}{2} - t$ yields $\frac{(r-1+\overline{\lambda}/2)^2}{2\overline{\lambda}} - \frac{\varepsilon^2\overline{\lambda}}{2}$. 
We thus find that $b_L(H) \geq \frac{(r-1+\overline{\lambda}/2)^2}{2\overline{\lambda}}-\frac{\overline{\lambda}}{2}$, from which the result follows. 
\end{proof}

We may again make use of Corollary~\ref{cor:Burning_equals_CMatching} to obtain a dual result. 

\begin{corollary}\label{app}
    Let $H$ be a hypergraph such that each pair of vertices occurs together in at most $\lambda$ hyperedges. Suppose that the minimum degree of the vertices is $\delta$, and $H$ has at least $\lceil \frac{\delta-1} {\lambda} +\frac{1}{2} \rceil$ hyperedges. We have that  
    \[
    b_L(H) \geq |V(H)| -|E(H)| +  \frac{\delta-1}{2 \lambda} - \frac{\delta-1}{2} - \frac{3\lambda}{8}.
    \]
\end{corollary}
In the case that the hypergraph is linear, meaning $\overline{\lambda}=1$, we may repeat the proof of Theorem~\ref{thm:BurningLowerUniformity}, but where we find that the function is maximized with either $t=r-1$ or $t=r$. This gives the following result. 

\begin{theorem}
    If $H$ is a linear hypergraph such that each hyperedge has cardinality at least $r$ and there are at least $r$ hyperedges, 
    then 
    $$b_L(H) \geq \binom{r}{2}.$$
\end{theorem}

Corollary~\ref{cor:Burning_equals_CMatching} again can be used when deriving the corresponding dual result. 

\begin{corollary}\label{cor:delta-delta}
If $H$ is a linear hypergraph with minimum degree $\delta$, and there are at least $\delta$ hyperedges, then 
    $$ 
    b_L(H) \geq |V(H)| -|E(H)| + \binom{\delta}{2}.$$
\end{corollary}

\subsection{Upper Bounds}

By decomposing a $C$-matching into the set of vertices or the set of edges that are contained in the matching, we find that we have a vertex cover or edge cover of $H$, respectively.  This also provides us with a number of upper bounds on the lazy burning number of $H$. 

\begin{lemma} \label{lem:Cmatching_MinEdgeCover}
If $H$ is a hypergraph, then we have that any maximum $C$-matching has cardinality at least that of a minimum vertex cover. 
\end{lemma}
\begin{proof}
    Let $M=(v_1h_1, v_2h_2, \ldots, v_kh_k)$ be a $C$-matching of the maximum possible length $k=|V(H)|-b_L(H)$. 
    We will show that the set of vertices given by $M_V=\{v_1, v_2,\ldots, v_k\}$ forms a vertex cover of $H$. 
    Suppose for contradiction that $M_V$ does not form a vertex cover. There must then be some hyperedge $h$ with none of its vertices in $M_V$. 
    For any vertex $v\in h$, the matched pair $vh$ may be appended to the $C$-matching to form a longer $C$-matching, contradicting maximality. 
\end{proof}

The proof to show that the cardinality of a maximum $C$-matching has at least that of the cardinality of a minimum edge cover of $H$ is analogous with dual terms interchanged. An \emph{isolated vertex} is not in any hyperedge. 
\begin{lemma} \label{lem:Cmatching_MinVertexCover}
If $H$ is a hypergraph with no isolated vertices, then any maximum $C$-matching has a cardinality at least the cardinality of a minimum edge cover of $H$. 
\end{lemma}

There are straightforward upper bounds on the minimum size of a vertex cover and an edge cover in terms of the maximum degree $\Delta(H)$ and maximum hyperedge cardinality $\overline{\Delta}(H)$, respectively. 
\begin{theorem}\label{upp}
If $H$ is a hypergraph with maximum vertex degree $\Delta$, then $$b_L(H) \leq |V(H)| - \left\lceil \frac{|E(H)|}{\Delta}\right\rceil .$$
\end{theorem}
\begin{proof}
We find a lower bound on the size of a vertex cover of $H$, which combines with Lemma~\ref{lem:Cmatching_MinVertexCover} to complete the proof. 
If a vertex cover of $H$ contains $t$ vertices, then since each vertex is contained in at most $\Delta$ hyperedges, at most $t\Delta$ hyperedges contain vertices from the vertex cover. 
Since each hyperedge of $E(H)$ contains a vertex in the cover, we then have $|E(H)| \leq t\Delta$.
We then have that $b_L(H) = |V(H)| - m(H) \leq |V(H)| - t$, and the result follows. 
\end{proof}

In an analogous fashion, the hyperedges from a $C$-matching form an edge cover of $H$, from which we derive the following result using an edge cover instead of a vertex cover. 
\begin{theorem}
If $H$ is a hypergraph with maximum hyperedge cardinality $\overline{\Delta}$ that contains no isolated vertices, then $$b_L(H) \leq |V(H)| - \left\lceil \frac{|V(H)|}{\overline{\Delta}}\right\rceil .$$
\end{theorem}

In the above two results, we found that the set of vertices (respectively, edges) contained within a $C$-matching formed a vertex (respectively, edge) cover of $H$. 
Instead, if we consider the ordered list of vertices $(v_1, v_2, \ldots, v_k)$ coming from the restriction of a maximum $C$-matching to just its vertices, then we arrive at a topic previously studied by the present authors in the context of burning Latin square hypergraphs; see \cite{LS}. 
An $n$-uniform Latin square hypergraph, written $H_L$, has vertex set as the entries of the Latin square $L$ and hyperedges containing all vertices that share a given row, a given column, or a given symbol. 
A \emph{cover-sequence} of $H_L$ is a sequence of vertices $v_1, v_2,\ldots, v_k$ such that $\{v_1, v_2,\ldots, v_k\}$ is a vertex cover of $H_L$ and each $v_i$ does not share all three of its incident hyperedges with vertices that proceeded it in the sequence. 
It is straightforward to see that the restriction of a $C$-matching to its vertices satisfies these conditions. 
Further, if we match a vertex $v_i$ to a hyperedge $h_i$ such that $h_i$ does not contain a vertex earlier in the sequence, then the constructed sequence of vertex-edge pairs forms a $C$-matching on $\mathrm{IG}(H_L)$.
As such, $C$-matchings of $\mathrm{IG}(H_L)$ are equivalent to cover-sequences of Latin squares, and so results in this section generalize results on the lazy burning of Latin squares in \cite{LS}.  

\section{Hypergraph Cores}

We introduce another characterization of lazy burning sets in terms of certain induced subhypergraphs. We are interested in the maximum induced subhypergraph so that every hyperedge is of cardinality greater than one. In the literature, this would correspond to the so-called 2-core of the dual hypergraph. For ease of notation, we refer to this simply as the \emph{core} of $H$, written $\core(H)$ (this is not to be confused with the core of a graph as used in graph homomorphism theory). The core can be obtained by continually deleting every hyperedge with cardinality one, along with the vertex it contains, until all remaining hyperedges have cardinality greater than one. 

Let $H$ be a hypergraph and $U \subseteq V(H)$. 
The \emph{subhypergraph of $H$ weakly induced by $U$}, denoted $H[U]$, is the hypergraph with $V(H[U]) = U$ and
$$ E(H[U]) = \{ h\cap U: h\in E(H)\text{ and } h\cap U\neq\emptyset\}. $$
This is also not to be confused with the strongly induced subhypergraphs; in the following, we mean an {\em induced subhypergraph} as a weakly induced subhypergraph since it is the only type of subhypergraph we are concerned with in this paper. We define the subhypergraph formed by vertex removal as $H\setminus U = H[V(H)\setminus U]$. 

We now provide an algorithm for determining the core of a hypergraph and a theorem on its correctness.
\setlength{\algomargin}{20pt}
\begin{algorithm}[!htbp]
\SetKwInOut{Input}{input}\SetKwInOut{Output}{output}
\Input{A hypergraph $H$.}
\Output{The core of $H$, an ordered list of vertex removals $R$.}
\BlankLine
Set $H' \leftarrow H$;\\ 
Set $R \leftarrow []$ (an empty list);\\
Set $S \leftarrow \{e\in E(H'): |e| = 1\}$;\\
\While{$S\neq\emptyset$}{
    Pick any $\{v\}\in S$; \\ 
    Set $S\leftarrow S\setminus\{\{v\}\}$;\\
    Add $v$ to $R$;\\ 
    Set $S \leftarrow S\cup\{\{u\}: \{u,v\}\in E(H')\}$;\\
	Set $H' \leftarrow H'\setminus\{v\}$; 
}
\Return{$H'$ and $R$.}
\caption{An algorithm returning the core of $H$ and an ordered list of vertex removals.}
\label{core_H}
\end{algorithm}

\begin{theorem}
After inputting a hypergraph $H$ to Algorithm~\ref{core_H}, it returns $\core(H)$ and a vertex set $R = V(H)\setminus V(\core(H))$ with an indexing. 
\end{theorem}
\begin{proof}
Suppose there are two induced subhypergraphs of $H$, say $H_1$ and $H_2$, such that all their hyperedges are of cardinality greater than one. The subhypergraph $H[V(H_1)\cup V(H_2)]$ also has the cardinality of all its hyperedges greater than one. 
Therefore, the core of a hypergraph is unique. 

From Algorithm~\ref{core_H}, we obtain an induced subhypergraph $J$ of $\core(H)$ with no hyperedge of cardinality one. 
Let $R = V(H)\setminus V(J) = \{v_1,v_2,\dots, v_r\}$ be the removed vertices, indexed by the algorithm. 
Suppose to the contrary that $J$ is not the core, and hence there is $A \subseteq R$ so that $H' = H[V(J)\cup A] = \core(H)$. Let $m$, with $1\leq m\leq r,$ be the least index so that $v_{m}\in A$. 
If $R' = \{v_1,v_2,\dots, v_{m-1}\}$, then $R'\subseteq R\setminus A$ and, therefore, $H'$ is an induced subhypergraph of $H\setminus R'$. 
By Algorithm~\ref{core_H}, we have that $\{v_{m}\}\in E(H\setminus R')$, which implies that $\{v_{m}\}\in E(H')$. 
This contradicts the assumption that $H'$ contains no singleton hyperedge. 
\end{proof}

We note that Algorithm 1 produces the core and a list of vertex removals in polynomial time. 
\begin{theorem}\label{thm: core complexity}
If $H$ is a hypergraph with order $n$ and $m$ edges, then the complexity of Algorithm~\ref{core_H} is $O(nm)$. 
\end{theorem}
\begin{proof}
Line~3 takes $O(m)$-time. 
The \textbf{while} loop on Line~4 repeats for at most $n$ times, say $O(n)$. 
For each iteration of the \textbf{while} loop, both Line~8 and 9 can be computed in $O(m)$-time. 
All other lines take $O(1)$-time. 
\end{proof}

In the following lemma, we collect some properties of cores. The proofs either follow from the definitions or are straightforward, arguing via a chronological list of lazy burning. A hypergraph is {\em degenerate} if its core contains no vertex.  

\begin{lemma}\label{lemmac}
Let $H$ be a hypergraph and $R =\{r_1,r_2,\dots, r_k\} = V(H)\setminus V(\core(H))$ be the vertex set returned by Algorithm~\ref{core_H}. Suppose $R\neq\emptyset$. We have the following. 
\begin{enumerate}
\item
For every $1\leq i\leq k$, $\{r_i\}$ is a singleton hyperedge in $H\setminus\{r_1,r_2,\dots, r_{i-1}\}$. 
In particular, $R\neq\emptyset$ implies $H$ contains a singleton hyperedge. 

\item If $L\subsetneq L'\subseteq V(H)$ then  $\core(H\setminus L')$ is an induced subhypergraph of $\core(H\setminus L)$. 
If $L' = L\cup\{v\}$ for some $v\in V(H)\setminus L$ such that $\{v\}\in E(H\setminus L)$, then $\core(H\setminus L') = \core(H\setminus L)$, and $H\setminus L'$ is degenerate if and only if $H\setminus L$ is degenerate.  
\item For all $1\leq i\leq k$, $\core(H) = \core(H\setminus\{r_1,r_2,\dots, r_i\})$. 
\end{enumerate}
\end{lemma}

We also need the following lemma.

\begin{lemma}\label{lemma: backward closure}
Let $H$ be a hypergraph, $h\in E(H)$, and $B\subseteq V(H)$. 
Assume $h\subseteq B$ and choose any vertex $u\in h$. 
The vertex set $B$ is a lazy burning set for $H$ if and only if $B\setminus\{u\}$ is a lazy burning set for $H$.
\end{lemma}
\begin{proof}
Fix a chronological list $S$ of $B$ and insert $\{h,u\}$ at the beginning of $S$. One may verify that the new list is a chronological list of a lazy burning set $B\setminus\{u\}$. 
The reverse implication follows immediately from the fact that $B\setminus \{u\}\subseteq B$.
\end{proof}

We provide another characterization of lazy burning in the following, which is the main result of the section.
\begin{theorem}\label{thm: empty core}
A vertex subset $B$ of a hypergraph $H$ is a lazy burning set if and only if the hypergraph $H\setminus B$ is degenerate.  
\end{theorem}
\begin{proof}
Let $B\subseteq V(H)$. 
We proceed by induction on the order of $H\setminus B$, say $k = |V(H\setminus B)|$. 
The statement holds when $k=0$, where $V(H) = B$. 

For the induction hypothesis, fix $k\geq 1$, and assume the following holds for any hypergraph $H_0$ and vertex set $B_0\subseteq V(H_0)$ satisfying $|V(H_0\setminus B_0)|< k$: $B_0$ is a lazy burning set for $H_0$ if and only if $H_0\setminus B_0$ is degenerate. Let $H$ be a hypergraph and $B\subseteq V(H)$ such that $|V(H\setminus B)|=k$. We will prove the statement holds for $H$ and $B$.

For the forward direction, suppose $B$ is a lazy burning set for $H$. Fix a chronological list of $B$ and let $\{v_1, h_1\}$ be the first element in the list. 
Let $B_1 = B\cup\{v_1\}$; we have that $B_1$ is a lazy burning set for $H$ since $B\subseteq B_1$. 
By the induction hypothesis and that $|V(H\setminus B_1)|<k$, we have that $H\setminus B_1$ is degenerate. 
Finally, since $\{v_1\} = h_1\setminus B$ is a hyperedge in $H\setminus B$, we have that $H\setminus B$ is degenerate by Lemma~\ref{lemmac} (2).

For the reverse direction, suppose $H\setminus B$ is degenerate. 
By Lemma~\ref{lemmac} (1), let $r_1\in R$ be such that $\{r_1\}\in E(H\setminus B)$.
This implies that there is a hyperedge $h\in E(H)$ so that $\{r_1\} =  h\setminus B$. 
By Lemma~\ref{lemmac} (2), $H\setminus(B\cup\{r_1\})$ is degenerate. Therefore, by the induction hypothesis, $B\cup \{r_1\}$ is a lazy burning set for $H$. Since $h\subseteq B\cup\{r_1\}$, $B$ is a lazy burning set for $H$ by Lemma~\ref{lemma: backward closure}.
\end{proof}

The next result reduces the burning number of a hypergraph to the burning number of its core.

\begin{theorem}\label{thm: also burns the core}
    For any hypergraph $H$ and $B\subseteq V(\core(H))$, $B$ is a lazy burning set for $\core(H)$ if and only if it is a lazy burning set for $H$.  
\end{theorem}

\begin{proof}
We proceed with the proof by induction on the order of $R = V(H)\setminus V(\core(H))$. 
When $|R| = 0$, we have $H = \core(H)$, and the statement holds. 

Let $H$ be a hypergraph so that $|R|>0$ and suppose the statement holds for all hypergraphs $H_0$ such that $|V(H_0)\setminus V(\core(H_0))|<|R|$. 
Assume an indexing on $R$ by Algorithm~\ref{core_H}. The first element $r_1\in R$ is a singleton hyperedge according to Lemma~\ref{lemmac} (1), so $\core(H) = \core(H\setminus\{r_1\})$ by Lemma~\ref{lemmac} (3). 
We then have that $$|R\setminus\{r_1\}| = |V(H\setminus\{r_1\})\setminus V(\core(H))|<|R|.$$ By the induction hypothesis, $B$ is a lazy burning set for $\core(H)$ if and only if $B$ is a lazy burning set for $H\setminus\{r_1\}$.

We finish the proof by a series of equivalent statements. The set $B$ being a lazy burning set for $\core(H)$ is equivalent to $B$ being a lazy burning set for $\core(H\setminus\{r_1\})$ because they are the same hypergraph according to Lemma~\ref{lemmac} (3). 
The latter is equivalent to $H\setminus(B\cup \{r_1\})$ being degenerate by Theorem~\ref{thm: empty core}, which is equivalent to $H\setminus B$ being degenerate by applying Lemma~\ref{lemmac} (2) with $L = B$ and $v = r_1$. 
The latter condition is equivalent to $B$ being a lazy burning set for $H$ by Theorem~\ref{thm: empty core}. The proof follows.
\end{proof}

As an immediate corollary, we have the following.

\begin{corollary}
\label{H_equals_core}
For any hypergraph $H$, $b_L(H) = b_L(\core(H))$. 
\end{corollary}

By applying hypergraph cores, we show the monotonicity of lazy burning numbers with respect to vertex removal. 
The following theorem contrasts with zero forcing, which is not monotone. 

\begin{theorem}
The lazy burning number is monotone under vertex removal; that is, for any vertex $v\in V(H)$, $b_L(H)-1\leq b_L(H\setminus\{v\})\leq b_L(H).$
\end{theorem}
\begin{proof}
For the upper bound, let $B$ be an optimal lazy burning set for $H$. 
Suppose first that $v\notin B$. Observe that $B\cup\{v\}$ is a lazy burning set for $H$, so $H\setminus(B\cup\{v\})$ is degenerate by Theorem \ref{thm: empty core}. Also note that $H\setminus(B\cup\{v\})=(H\setminus\{v\})\setminus B$ and $B\subseteq V(H\setminus\{v\})$, so we have that $B$ is a lazy burning set for $H\setminus\{v\}$ by Theorem \ref{thm: empty core}. We can therefore conclude that $b_L(H\setminus\{v\})\leq |B|=b_L(H)$ in this case.

Now, suppose $v\in B$. Observe that $H\setminus B=(H\setminus\{v\})\setminus (B\setminus\{v\})$, and hence, they have the same core. By Theorem~\ref{thm: empty core}, $H\setminus B$ is degenerate, so $(H\setminus\{v\})\setminus (B\setminus\{v\})$ is also degenerate. This implies $B\setminus\{v\}$ is a lazy burning set for $H\setminus\{v\}$ by Theorem \ref{thm: empty core}. Hence, $b_L(H\setminus\{v\})\leq |B\setminus\{v\}|<|B|=b_L(H)$ in this case.

For the lower bound, let $B$ be an optimal lazy burning set for $H\setminus\{v\}$. 
By Theorem~\ref{thm: empty core}, we have that $(H\setminus\{v\})\setminus B$ is degenerate. Observe that $(H\setminus\{v\})\setminus B = H\setminus(B\cup\{v\})$. Since this hypergraph is degenerate,  $B\cup\{v\}$ is a lazy burning set for $H$ by Theorem~\ref{thm: empty core}. We therefore have $b_L(H)\leq |B\cup\{v\}|=|B|+1=b_L(H\setminus\{v\})+1$, and the lower bound follows.\end{proof}

\section{Complexity via Zero Forcing}

The principal goal of this section is to show that computing an upper bound on the lazy burning number is NP-complete, which answers a conjecture in \cite{BJP}.  We first provide another characterization of lazy burning sets in hypergraphs via zero forcing sets in their incidence graphs.

\begin{theorem}\label{zero_set_equivalence}
For a hypergraph $H$, a subset $B\subseteq V(H)$ is a lazy burning set for $H$ if and only if $B\cup E(H)$ is a zero forcing set for $\mathrm{IG}(H)$.
\end{theorem}
\begin{proof}
For the forward direction, suppose that $B$ is a lazy burning set for $H$. 
Let $k = |V(H)| - |B|$ and consider a chronological list of lazy burning of $S = (\{v_1,h_1\},\{v_2,h_2\},\dots, \{v_k, h_k\})$ with the corresponding list of burned vertices $B = B_0, B_1,\dots, B_k = V(H)$.

For $0\leq i\leq k$, define $B_i' = B_i\cup E(H)\subseteq V(\IG(H))$.
For $1\leq i\leq k$, we have that $h_i\setminus\{v_i\}\subseteq B_{i-1}$. 
It follows that $N_{\IG(H)}(h_i)\setminus\{v_i\}\subseteq B_{i-1} \subseteq B_{i-1}'$ in $\IG(H)$, and hence, the list $(h_1\to v_1, h_1\to v_2, \dots, h_k\to v_k)$ is a chronological list of zero forcing in $\IG(H)$, where $B_0', B_1', \dots, B_k'$ is the corresponding list of black vertices. 
We may conclude that $B\cup E(H)$ is a zero forcing set for $\IG(H)$. 

The reverse direction follows analogously by reversing the implications in the forward direction above. 
\end{proof}

We note that Theorem~\ref{zero_set_equivalence} gives the new inequality $$z(\mathrm{IG}(H))\leq b_L(H)+|E(H)|.$$

Denote the lazy burning number as defined in \cite{JT}, where singleton hyperedges do not spontaneously burn, by $b_L^o(H)$.  Consider the following decision problems. 

\medskip

\noindent PROBLEM: Lazy burning   \\
INSTANCE: A hypergraph $H$ and a positive integer $k \leq |V(H)|$.\\ 
QUESTION: Is $b_L(H)\leq k$?

\medskip

\noindent PROBLEM: Lazy burning without spontaneous burning  \\
INSTANCE: A hypergraph $H$ and a positive integer $k \leq |V(H)|$.\\ 
QUESTION: Is $b_L^o(H)\leq k$?

\medskip

To show the NP-completeness of the two problems, we need the following theorem. 

\begin{theorem}\label{thm 0}
For every hypergraph $H$, if $H$ contains no singleton hyperedge, then $b_L^o(H) = b_L(H)$. Further, $b_L^o(\core(H)) = b_L(H)$. 
\end{theorem}
\begin{proof}
For the first statement, if a list is a chronological list of lazy burning in one process, then it is also a chronological list of lazy burning in the other one. For the second statement, we have that $b_L^o(\core(H)) = b_L(\core(H)) = b_L(H)$, where the first equality holds since the core of a hypergraph contains no singleton hyperedge, and the second equality holds by Corollary~\ref{H_equals_core}.
\end{proof}

The symmetric and skew zero forcing decision problems are as follows. 

\medskip

\noindent PROBLEM: Zero forcing \\
INSTANCE: A graph $G$ and a positive integer $k \leq |V(G)|$.\\ 
QUESTION: Is $z(G)\leq k$?

\medskip

\noindent PROBLEM: Skew zero forcing \\
INSTANCE: A graph $G$ and a positive integer $k \leq |V(G)|$.\\ 
QUESTION: Is $z_0(G)\leq k$?

\medskip

Let $G$ be a graph and $N(v) = N_G(v)$ be the \textit{neighborhood} of $v\in V(G)$. 
The {\em open neighborhood hypergraph} ${\calN}(G)$ is a hypergraph with $V({\calN}(G)) = V(G)$ and $E(\calN(G)) = \{N_G(v):v\in V(G)\}$. 
The {\em closed neighborhood hypergraph ${\calN}[G]$} is defined analogously with the hyperedge set consisting of the {\em closed neighborhood} $N[v] = N_G[v] = N(v)\cup\{v\}$ of each vertex in $G$. 

The following theorem characterizes skew zero forcing as lazy burning on the open neighborhood hypergraph.

\begin{theorem}
\label{skew_forcing_equals_lazy_burning}
Let $G$ be a graph and $B\subseteq V(G)$. The set $B$ is a skew zero forcing set for $G$ if and only if $B$ is a lazy burning set for $\calN(G)$. In particular, $z_0(G) = b_L(\calN(G))$. 
\end{theorem}

\begin{proof}
Let $\calB = (B = B_0, B_1, \dots, B_k = V(G))$ be a list of vertex subsets, where $k = |V(H)| - |B|$, and $V(H)\setminus B = \{v_1, v_2, \dots, v_k\}$ is indexed so that $B_i = B_{i-1}\cup \{v_i\}$ for $1\leq i\leq k$. 

Suppose that $\calB$ is the corresponding list of black vertices of a chronological list of skew zero forces of $B$, say $S = (u_1\to v_1, u_2\to v_2, \dots, u_k\to v_k)$. 
For $1\leq i\leq k$, we have that $N(u_i)\setminus\{v_i\}\subseteq B_{i-1}$.  
Hence, $S$ is a chronological list of skew zero forcing if and only if the list $$(\{v_1, N(u_1)\}, \{v_2, N(u_2)\},\dots, \{v_k, N(u_k)\})$$ is a chronological list of lazy burning of $B$ on $\calN(G)$.
\end{proof}

An analogous and so omitted argument proves the following.

\begin{theorem}\label{thm: closed neighborhood and zero forcing}
Let $G$ be a graph and $B\subseteq V(G)$. 
If $B$ is a zero forcing set for $G$, then $B$ is a lazy burning set for $\calN[G]$. 
In particular, $b_L(\calN[G])\leq z(G)$. 
\end{theorem}

The converse of Theorem~\ref{thm: closed neighborhood and zero forcing} does not hold. Let $K_{1,j}$ be the star graph, where $V(K_{1,j}) = \{v,u_1, u_2,\dots, u_j\}$ and $v$ is the universal vertex. 
Note that $N_{K_{1,j}}[v] =  V(K_{1,j})$ and $N_{K_{1,j}}[u_i] = \{v, u_i\}$ for each leaf $u_i$. We then have that  
$$E(\calN[K_{1,j}]) = \{V(K_{1,j}), v u_1, v u_2, \dots, v u_j\}.$$ 
The set $\{v\}$ is a lazy burning set for ${\calN}[K_{1,j}]$ since burning spreads via each edge $vu_i\in E({\calN}[K_{1,j}])$. 
Meanwhile, the zero forcing number of a star graph is $j -1 $ by initially forcing all leaves except one. 
 
For completeness, we restate results on the complexity of skew zero forcing from \cite{CF}. We first need some terminology introduced there. For a graph $G$, let $T(G)$ denote the graph obtained by ``gluing a triangle'' to each vertex of $G$. 
More precisely, $V(T(G)) = V(G)\times\{1,2,3\}$ and 
\begin{equation*}
\begin{split}
E(T(G)) = &\ \{e\times\{1\}: e\in E(G)\}\ \cup \\
 &\ \{(v,i)(v, j): v\in V(G), i,j\in\{1,2,3\}, i\neq j\}.     
\end{split}
\end{equation*}
In particular, the induced subgraph $T(G)[V(G)\times\{1\}]$ is isomorphic to $G$ using a canonical map such that $(v,1)\mapsto v$, for $v\in V(G)$. 
\begin{theorem}[{\cite[Theorem~3.7]{CF}}]\label{red}
For a graph $G$ and $B\subseteq V(G)$, $B$ is a zero forcing set if and only if $B\times\{1\}$ is a skew zero forcing set for $T(G)$. Further, $B$ is optimal if and only if $B\times\{1\}$ is optimal; that is $z(G) = z_0(T(G))$. 
\end{theorem}

By Theorem~\ref{red}, the skew zero forcing problem is at least as hard as the zero forcing problem, which is shown to be NP-complete in \cite{A, Y}. We then have the following result.
\begin{theorem}[{\cite[Corollary~3.8]{CF}}]\label{szf is NPC}
The skew zero forcing problem is NP-complete.
\end{theorem}

We now state the main result of this section.

\begin{theorem}
\label{lbsc is npc}
The lazy burning problem is NP-complete. 
\end{theorem}
\begin{proof}
Given a hypergraph $H$ and a vertex set $B\subseteq V(H)$, determining whether or not $B$ is a lazy burning set for $H$ takes polynomial time by determining whether or not $H\setminus B$ is degenerate, 
as in Theorem~\ref{thm: empty core}. 
Theorem~\ref{thm: core complexity} states that Algorithm~\ref{core_H} can verify it in polynomial time. Hence, the lazy burning problem is in NP.

By Theorem~\ref{skew_forcing_equals_lazy_burning}, we know that $\calN(\cdot)$ is a transformation from the skew zero forcing problem to the lazy burning problem. That is, if a vertex set $B\subseteq V(G)$ is a skew forcing a set of a graph $G$, then $B\subseteq V(\calN(G)) = V(G)$ is a lazy burning set for the hypergraph $\calN(G)$. 
Note that $\calN(\cdot)$ is a polynomial-time transformation. 
Therefore, the lazy burning problem is NP-hard by Theorem~\ref{szf is NPC}, and hence, is NP-complete. 
\end{proof}

We also derive the following.

\begin{theorem}
The lazy burning problem without spontaneous burning is NP-complete. 
\end{theorem}
\begin{proof}
Given a hypergraph $H$ and a vertex subset $B$, we first show that verifying whether or not $B$ is a lazy burning set can take polynomial time. This can be done by considering first taking another hypergraph $H'$ with $V(H') = V(H)$ and $E(H') = E(H)\setminus\{e\in E(H): |e|=1\}$. We know that $B$ is a lazy burning set without spontaneous burning if and only if $ H'\setminus B$ is degenerate. The first step takes $O(m)$-time, and the second step takes $O(nm)$-time by Theorem~\ref{thm: core complexity}.  In particular, the lazy burning problem without spontaneous burning is in NP.

The function $\core(\cdot)$ is a transformation from the lazy burning problem with spontaneous burning to the lazy burning problem without spontaneous burning, which is proved in Theorem~\ref{thm 0}. Hence, the lazy burning problem without spontaneous burning is NP-hard and, hence, is NP-complete. 
\end{proof}

We may use the complexity results from lazy burning to analyze the complexity of the skew zero forcing numbers of bipartite graphs. 
Given a bipartite graph $G = (X\cup Y, I)$, we may construct two hypergraphs $H_1$ and $H_2$, where $V(H_1) = X$, $E(H_1) = \{N_G(y): y\in Y\}$, and $H_2 = H_1^*$; that is, we may view every bipartite graph as a hypergraph $H = (V, E)$ by considering $V = X$ and $E = \{N_G(y): y\in Y\}$. 
Note that the hyperedge set $E=E(H_1)$ may be a multiset, and lazy burning performs identically on hypergraphs whenever $E(H_1)$ is a set or a multiset. 
\begin{theorem}\label{thm: lb-lbdual is szf}
Let $H$ be a hypergraph. We then have that $B$ is a lazy burning set for $H$ and $B^*$ is a lazy burning set for $H^*$ if and only if $B\cup B^*$ is a skew zero forcing set for $\IG(H)$. 
In particular, 
$$ z_0(\IG(H)) = b_L(H) + b_L(H^*). $$
\end{theorem}
\begin{proof}
The proof is straightforward by considering that the neighborhood hypergraph of an incidence graph $G = \IG(H)$ is precisely the disjoint union of $H\cup H^*$. 
The statement follows by Theorem~\ref{skew_forcing_equals_lazy_burning}. 
\end{proof}
We have the following by the alternative form of Corollary~\ref{cor1}. 
\begin{corollary}\label{cor: z=2bl}
For any hypergraph $H$, 
$ z_0(\IG(H)) = 2b_L(H) + |E(H)| - |V(H)|$.
\end{corollary}

\medskip\noindent
PROBLEM: Skew zero forcing on bipartite graphs\\
INSTANCE: A bipartite graph $G$ and an integer $k\leq |V(G)|$. \\
QUESTION: Is $z_0(G)\leq k$? 
\medskip

We arrive at the following new result on skew zero forcing. The results in \cite{DeAlba} also give a similar characterization of skew zero forcing on bipartite graphs using special matchings, but our technique is essentially different. 
\begin{theorem}\label{szf on bpt is npc}
The skew zero forcing problem on bipartite graphs is NP-complete. 
\end{theorem}

\begin{proof}
For a hypergraph $H$, it takes polynomial time to generate the bipartite graph $\IG(H)$. By Corollary~\ref{cor:Burning_equals_CMatching}, from a lazy burning set $B$ for $H$, we obtain a $C$-matching by performing the lazy burning process and then obtain a lazy burning set for $H^*$.

Since verifying if a set of vertices is a skew zero forcing set for a graph takes polynomial time, and the lazy burning problem is NP-complete by Theorem~\ref{lbsc is npc}, we have that the skew zero forcing problem on bipartite graphs is NP-complete. 
\end{proof}

\section{Further Directions}

We presented several characterizations of lazy burning on hypergraphs via matchings and zero forcing. It would be interesting to see what new insights can be provided by these connections, especially in the case of zero forcing on graphs. The lazy burning of hypergraphs associated with Steiner Triple Systems was shown to equal their dimension; see 
\cite{STS}.  One question is to determine if there are relationships between the lazy burning number of hypergraphs associated with other designs, such as balanced incomplete block designs, and a natural notion of dimension. 

A hypergraph is {\em critical $k$-burnable} if $b_L(H) = k$, and removing any vertex results in the decrease of its lazy burning number. It would be interesting to characterize critical $k$-burnable hypergraphs. 

\section{Acknowledgements}
The first author was supported by an NSERC grant, and the second author was supported by an NSERC PGS-D scholarship.

\end{document}